\documentclass[12pt,twoside]{amsart}
\usepackage{amsmath}

\usepackage{amsthm}
\usepackage{amsfonts}
\usepackage{amssymb}
\usepackage{latexsym}
\usepackage{mathrsfs}
\usepackage{amsmath}
\usepackage{amsthm}
\usepackage{amsfonts}
\usepackage{amssymb}
\usepackage{latexsym}
\usepackage{geometry}
\usepackage{dsfont}
\usepackage{graphicx}
\usepackage{color}
\usepackage[all]{xy}

\date{}
\allowdisplaybreaks[4] \footskip=15pt
\renewcommand{\uppercasenonmath}[1]{}

\usepackage{graphicx,amssymb}
\usepackage[all]{xy}
\usepackage{amsmath}

\allowdisplaybreaks
\usepackage{amsthm}
\usepackage{color}

\theoremstyle{plain}

\theoremstyle{plain}
\newtheorem{theorem}{Theorem}[section]
\newtheorem{proposition}[theorem]{Proposition}
\newtheorem{lemma}[theorem]{Lemma}
\newtheorem{corollary}[theorem]{Corollary}
\newtheorem{example}[theorem]{Example}
\newtheorem*{open question}{Open Question}

\theoremstyle{definition}
\newtheorem*{acknowledgement}{Acknowledgement}

\theoremstyle{remark}

\newcommand{\Pp}{\mathcal{P}}

\newcommand{\Tor}{\mbox{\rm Tor}}
\newcommand{\im}{\mbox{\rm Im}}

\newcommand{\Prufer}{Pr\"{u}fer}

\newcommand{\Q}{\mathcal{Q}}

\def\tor{{\rm tor}}
\def\Hom{{\rm Hom}}
\def\Ext{{\rm Ext}}
\def\Tor{{\rm Tor}}

\def\Ker{{\rm Ker}}

\def\Im{{\rm Im}}
\def\Cok{{\rm Cok}}
\def\im{{\rm im}}

\def\Max{{\rm Max}}

\def\WQ{{\rm WQ}}
\def\T{{\rm T}}
\def\E{{\rm E}}
\def\DQ{{\rm DQ}}
\def\Spec{{\rm Spec}}
\newcommand{\m}{\frak{m}}
\newcommand{\p}{\frak{p}}

\newcommand{\Lu}{\mathcal{L}}

\begin{document}
\begin{center}
{\large  \bf  Some remarks on Lucas modules}

\vspace{0.5cm}  Xiaolei Zhang$^{a,\dag}$,\ Guocheng Dai$^{b}$,\   Wei Qi$^{a}$\\

{\footnotesize a. School of Mathematics and Statistics, Shandong University of Technology, Zibo 255000, China\\

c.\ School of Mathematics Sciences, Sichuan Normal University,  Chengdu 610066, China\\

$\dag$.\ Corresponding author. E-mail addresses: (Xiaolei Zhang) zxlrghj@163.com\\}
\end{center}

\bigskip
\centerline { \bf  Abstract}
\bigskip
\leftskip10truemm \rightskip10truemm \noindent

In this paper, we discuss some properties on Lucas modules. In details, we show that direct and inverse limits of Lucas modules are  Lucas modules, and every $R$-module has a Lucas envelope and a Lucas cover. Moreover, some properties of direct and inverse limits of Lucas modules and some constructions and the unique mapping properties of Lucas envelopes and Lucas covers are investigated.
\vbox to 0.3cm{}

\leftskip0truemm \rightskip0truemm
\bigskip

\section{Introduction and Preliminaries}
Throughout this paper, we always assume $R$ is a commutative ring with identity. Recall that an ideal $I$ of $R$ is said to be \emph{dense} if $(0:_RI):=\{r\in R\mid Ir=0\}=0$, and \emph{semi-regular} if there exists a finitely generated dense sub-ideal of $I$. The set of all finitely generated semi-regular ideals of $R$ is denoted by $\Q$.  For a ring $R$, we denote by $\T(R)$ the total quotient ring of $R$. We denote by $\sharp(S)$ the cardinal number of a set $S$.

The notion of Lucas modules was first introduced by Wang et al. \cite{wzcc20} to give an example of total rings with small finitistic dimension larger than zero, which provided a negative solution to the open question in \cite[Problem 1b]{CFFG14}. Actually, they showed that if a ring $R$ has  small finitistic dimension equal to zero, then $R$ satisfies that every $R$-module is a Lucas modules (see \cite[Theorem 3.9, Proposition 3.10]{wzcc20}). Furthermore, the converse was also proved to be true by Wang et al. \cite[Proposition 2.2]{fkxs20}.  For the general case, refer to \cite{z-fpd}. It is well-known that $w$-operation is one of most important star operations for studying commutative rings (see \cite{ywzc11}). For giving a star operation viewpoint  of Lucas modules, Zhou and coauthors \cite{ZDC20} introduced the notion of $q$-operations over commutative rings utilizing finitely generated semi-regular ideals. In 2016, Zhou and coauthors \cite{DF16} studied the direct and inverse limits of $w$-modules. The main motivation of this paper is to  discuss some  properties on direct and inverse limits of Lucas modules,   Lucas envelopes and  Lucas covers. So we will firstly recall the definitions on Lucas modules and $q$-operations. For more details on this topic,
see \cite{wzcc20,fkxs20,z-tq-f-c,ZDC20}.

Let $M$ be an $R$-module. Denote by
 \begin{center}
{\rm $\tor_{\Q}(M):=\{x\in M|Ix=0$, for some $I\in \Q \}.$}
\end{center}
$M$ is said to be \emph{$\Q$-torsion} (resp. \emph{$\Q$-torsion-free}) if $\tor_{\Q}(M)=M$ (resp. $\tor_{\Q}(M)=0$). Note that an $R$-module $M$ is $\Q$-torsion-free if and only if $\Hom_R(R/I,M)=0$ for any $I\in \Q$. Trivially, the class of $\Q$-torsion modules is closed under submodules, quotients and direct sums, and the class of $\Q$-torsion-free modules is closed under submodules and direct products. A $\Q$-torsion-free module $M$ is called a \emph{Lucas module} if $\Ext_R^1(R/I,M)=0$ for any $I\in \Q$. The class of all Lucas modules is denoted by $\Lu$. We have to be aware that $R$ itself may not be a Lucas module. The \emph{Lucas envelope} of a $\Q$-torsion-free module $M$ is given by
\begin{center}
{\rm $M_q:=\{x\in \E_R(M)|Ix\subseteq M$, for some $I\in \Q \},$}
\end{center}
where $\E_R(M)$ is the injective envelope of $M$ as an $R$-module.
\cite[Proposition 2.10]{wzcc20} pointed out that $M$ is a Lucas module if and only if $M_q=M$, and $M_q$ is the minimal Lucas module containing $M$. The following result shows that the Lucas modules over an integral domain are exactly vector spaces over its  quotient field.

\begin{proposition}\label{int-luc} Let $R$ be an integral domain with $Q$ its quotient field.  Then the following statements are equivalent for an $R$-module $M$.
\begin{enumerate}
\item  $M$ is a Lucas module.
\item $M$ is an injective torsion-free $R$-module.
\item $M$ is a vector space over $Q$.
\end{enumerate}
Consequently, if $M$ is a torsion-free module, then $M_q=\E_R(M)\cong M\otimes_RQ$.
\end{proposition}
\begin{proof} $(1)\Rightarrow (2)$ Let $M$ be a Lucas module. Since $R$ is an integral domain,  every finitely generated nonzero ideal is in $\Q$.  So an $R$-module is $\Q$-torsion-free if and only if it is torsion-free. Let $I$ be a nonzero ideal of $R$ and $x$ a nonzero element in $I$. Considering the following induced exact sequence $0=\Hom_R(I/\langle x\rangle,M)\rightarrow \Ext_R^1(R/I,M)\rightarrow \Ext_R^1(R/\langle x\rangle,M)=0$, we have $\Ext_R^1(R/I,M)=0$, and hence $M$ is an injective $R$-module.

$(2)\Rightarrow (1)$  is trivial.

$(2)\Leftrightarrow (3)$  follows by \cite[Theorem 1.1.7]{gt}.

Suppose $M$ is a torsion-free module. Then $M$  is $\Q$-torsion-free. So $M_q=\E_R(M)\cong M\otimes_RQ$ by \cite[Proposition 2.10]{wzcc20}.
\end{proof}

\begin{proposition}\label{c-l-p} Let $R$ be a ring. Then a $\Q$-torsion-free $R$-module $M$ is a Lucas module if and only if $\Ext_R^1(T,M)=0$ for any  $\Q$-torsion $R$-module $T$.
\end{proposition}
\begin{proof} Suppose a $\Q$-torsion-free $R$-module $M$ satisfies $\Ext_R^1(T,M)=0$ for any  $\Q$-torsion $R$-module $T$. Then $M$ is trivially a Lucas module. Suppose $M$ is a Lucas module.

{\bf Claim 1: that $\Ext_R^1(R/I,M)=0$ for any semi-regular ideal $I$ of $R$.} Note there is a ideal $J\in\Q$ such that $J\subseteq I$. Consider the following induced exact sequence $0=\Hom_R(I/J,M)\rightarrow \Ext_R^1(R/I,M)\rightarrow \Ext_R^1(R/J,M)=0$. So $\Ext_R^1(R/I,M)=0$.

Let $B$ be an $R$-module, $A$ a submodule of $B$ such that $B/A$ is $\Q$-torsion. Let $f:A\rightarrow E$ be an $R$-homomorphism. Set
\begin{center}
 $\Gamma=\{(C,d)|C$ is a submodule of $B$ containing $A$, $d:C\rightarrow E$ is an $R$-homomorphism satisfying $d|_A=f\}.$
\end{center}
Since $(A,f)\in \Gamma$, $\Gamma$ is nonempty. Set $(C_1,d_1)\leq (C_2,d_2)$ if and only if  $C_1\subseteq C_2$ and $d_2|_{C_1}=d_1$. Then $\Gamma$ is a partial order. For every chain $\{(C_j,d_j)\}$, let $C_0=\bigcup\limits_{j}C_j$, and if $c\in C_j$, then $d_0(c)=d_j(c)$. It is easy to verify $(C_0,d_0)$ a upper bound of the chain $\{(C_j,d_j)\}$. So by Zorn Lemma, there is a maximal element $(C,d)$ in $\Gamma$.

{\bf Claim 2:   $C=B$.} For contradiction, let $x\in B-C$. Set  $I=\{r\in R|rx\in C\}$. since $B/A$ is $\Q$-torsion, the sub-quotient $(Rx+C)/C\cong R/I$ is also  $\Q$-torsion. Hence $I$ is semi-regular ideal of $R$. Let $h:I\rightarrow E$ be an $R$-homomorphism satisfying $h(r)=d(rx)$. Then there is an $R$-homomorphism $g: R\rightarrow E$ such that $g(r)=h(r)=d(rx)$ $(\forall r\in I)$ by Claim 1. Let  $C_1=C+Rx$ and $d_1(c+rx)=d(c)+g(r)$, where $c\in C$ and $r\in R$.  If $c+rx=0$, then $r\in I$. So $d(c)+g(r)=d(c)+h(r)=d(c)+d(rx)=d(c+rx)=0$. Hence  $d_1$ is well-defined and $d_1|_A=f$. So $(C_1,d_1)\in  \Gamma$. However,  $(C_1,d_1)> (C,d)$ which is a contradiction to the maximality of  $(C,d)$.

Let $T$ be a $\Q$-torsion $R$-module. Consider the exact sequence $0\rightarrow A\xrightarrow{i} F\rightarrow T\rightarrow 0$ where $F$ is free. Then we have a long exact  sequence $\Hom_R(F,M)\xrightarrow{\Hom_R(i,M)} \Hom_R(A,M)\rightarrow \Ext^1_R(T,M)\rightarrow \Ext^1_R(F,M)=0$. By Claim 2, $\Hom_R(i,M)$ is an epimorphism. So  $\Ext^1_R(T,M)=0$.
\end{proof}

\begin{proposition}\label{l-p} Let $R$ be a ring. Then the following statements hold.
\begin{enumerate}
\item  every direct product $\prod M_i$ of $R$-modules is a Lucas module if and only if each $M_i$ is a Lucas module.
\item every pure submodule of a Lucas module is a Lucas module.
\end{enumerate}
\end{proposition}
\begin{proof} (1) Obvious.

(2) Let $I$ be a finitely generated semi-regular ideal of $R$, $M$ a Lucas module and $N$ a pure submodule of $M$. Then trivially, $N$ is $\Q$-torsion-free. Since $R/I$ is finitely presented, we have an exact sequence $0\rightarrow \Ext_R^1(R/I,N)\rightarrow \Ext_R^1(R/I,M)=0.$ So $\Ext_R^1(R/I,N)=0$. Hence $N$ is also a Lucas module.
\end{proof}

Recall from \cite{ZDC20} that a submodule $N$ of a  $\Q$-torsion free module $M$ is called a $q$-submodule if $N_q\cap M=N$. If an ideal $I$ of $R$ is a $q$-submodule of $R$, then $I$ is also called a $q$-ideal of $R$. A \emph{maximal $q$-ideal} is an ideal of $R$ which is maximal among the $q$-submodules of $R$. The set of all maximal $q$-ideals is denoted by $q$-$\Max(R)$. $q$-$\Max(R)$ is exactly the set of all maximal non-semi-regular ideals of $R$, and thus is non-empty and a subset of $\Spec(R)$ (see
\cite[Proposition 2.5, Proposition 2.7]{ZDC20}).

An $R$-homomorphism $f:M\rightarrow N$ is called to be  a \emph{$\tau_q$-monomorphism} (resp., \emph{$\tau_q$-epimorphism}, \emph{$\tau_q$-isomorphism}) provided that $f_\m:M_\m\rightarrow N_\m$ is a monomorphism (resp., an epimorphism, an isomorphism) over $R_\m$ for any $\m\in q$-$\Max(R)$. By \cite[Proposition 2.7(5)]{ZDC20}, an $R$-homomorphism $f:M\rightarrow N$ is a $\tau_q$-monomorphism (resp., $\tau_q$-epimorphism, $\tau_q$-isomorphism) if  and only if $\Ker(f)$ is  (resp.,  $\Cok(f)$ is, both $\Ker(f)$ and $\Cok(f)$ are) $\Q$-torsion. A sequence of $R$-modules $A\xrightarrow{f} B\xrightarrow{g} C$ is said to be  \emph{$\tau_q$-exact} provided that $A_\m\rightarrow B_\m\rightarrow C_\m$ is  exact as  $R_\m$-modules for any  $\m\in q$-$\Max(R)$. Following from \cite[Proposition 2.4]{z-tq-f-c} that an $R$-sequence $A\rightarrow B\rightarrow C$ is $\tau_q$-exact if and only if  $A\otimes_R\T(R[x])\rightarrow B\otimes_R\T(R[x])\rightarrow C\otimes_R\T(R[x])$ is  $\T(R[x])$-exact. Let $M$ be an $R$-module. Then
\begin{enumerate}
\item   $M$ is said to be \emph{$\tau_q$-finitely generated} provided that there exists  a $\tau_q$-exact sequence $F\rightarrow M\rightarrow 0$ with $F$ finitely generated free;

\item  $M$ is said to be \emph{$\tau_q$-finitely presented} provided that there exists a $\tau_q$-exact sequence $$ F_1\rightarrow F_0\rightarrow M\rightarrow 0$$ such that $F_0$ and $F_1$ are finitely generated free modules.
\end{enumerate}
Following from \cite[Theorem 3.3]{z-tq-f-c} that an $R$-module  $M$ is $\tau_q$-finitely generated (resp., $\tau_q$-finitely  presented) if and only if $M\otimes_R\T(R[x])$ is finitely generated (resp., finitely  presented)  over $\T(R[x])$.

\section{the inverse limits and enveloping properties of Lucas modules}

First, we will study the inverse limit of Lucas modules.
\begin{theorem}\label{invse lim lucas} Let $\{M_i, \varphi_{i,j}\mid i,j\in \Gamma\}$ be an inverse system of Lucas  modules over $R$. Then $\lim\limits_{\longleftarrow}M_i$ is also a Lucas module.
\end{theorem}
\begin{proof} Let $I\in \Q$ and $f:I\rightarrow \lim\limits_{\longleftarrow}M_i$ be an $R$-homomorphism. It is sufficient to show that $f$ can be extended to some  $R$-homomorphism $g:R\rightarrow \lim\limits_{\longleftarrow}M_i$. Let $\lambda:\lim\limits_{\longleftarrow}M_i\hookrightarrow \prod M_i$ be the natural embedding map.
By Proposition \ref{l-p}, $\prod M_i$ is also a Lucas module. So there is an $R$-homomorphism $h:R\rightarrow \prod M_i$ such that the following diagram is commutative:

$$\xymatrix@R=17pt@C=30pt{
&\prod M_i&\\
&\lim\limits_{\longleftarrow}M_i\ar@{^{(}->}[u]^{\lambda}&\\
0\ar[r]&I\ar[u]^{f}\ar[r]^{i}&R\ar@{-->}[luu]_{h}.\\
}$$
Set $h(1)=(m_i)\in \prod M_i$. We claim that $h(1)\in \lim\limits_{\longleftarrow}M_i$. Indeed, let $r\in I$. Then we have $f(r)=h(r)=rh(1)=r(m_i)\in\lim\limits_{\longleftarrow}M_i$. Then we have $rm_i=\varphi_{i,j}(rm_j)=r\varphi_{i,j}(m_j)$ for any $i\leq j$. So $I(m_i-\varphi_{i,j}(m_j))=0\in M_i$. Since $M_i$ is $\Q$-torsion-free, $m_i=\varphi_{i,j}(m_j)$ for any $i\leq j$. It follows from \cite[Theorem 1.5.10]{EJ00} that $h(1)\in \lim\limits_{\longleftarrow}M_i$. Hence $\Im(h)\subseteq \lim\limits_{\longleftarrow}M_i$. Setting $g:=h$, we have $f=h\circ i=g\circ i$.
\end{proof}

Let  $\mathcal{F}$ be a class of $R$-modules and $M$ an $R$-module. If there is a homomorphism $f: M\rightarrow F$ with  $F\in \mathcal{F}$ such that $\Hom_R(F,F')\rightarrow \Hom_R(M,F')$ is an epimorphism for any $F'\in \mathcal{F}$, then we say $M$ has an  $\mathcal{F}$-preenvelope and $f: M\rightarrow F$ is said to be an  $\mathcal{F}$-preenvelope of $M$. If, moreover, every endomorphism $h$ such that $f=h\circ f$  is an isomorphism,  then $f: M\rightarrow F$ is said to be an  $\mathcal{F}$-envelope of $M$.  Let $f: M\rightarrow F$  be an $\mathcal{F}$-envelope of $M$. If for any $f':M\rightarrow F'$ with $F'\in\mathcal{F}$, there exists a unique $g:F\rightarrow F'$ such that $f'=g\circ f$. Then $f$ is said to have the unique mapping property.

\begin{lemma}\cite[Lemma 5.3.12]{EJ00}\label{210}
{\rm
Let $M$ and $N$ be $R$-modules. Then there is a cardinal number $\kappa_{\alpha}$ dependent on $\sharp(N)$ and  $\sharp(R)$ such that for any morphism  $f: N\rightarrow M$, there is a pure submodule $S$ of $M$ such that $f(N)\subseteq S$ and  $\sharp (S)\leq \kappa_{\alpha}$.}
\end{lemma}

\begin{lemma}\cite[Corollary 6.2.2]{EJ00}\label{211}
{\rm
Let $\sharp (M)=\kappa_{\beta}$. Suppose there is an infinite cardinal $\kappa_{\alpha}$ such that if $F\in \mathcal{F}$ and $S\subseteq F$ is a submodule with $\sharp (S)\leq \kappa_{\beta}$, there is a submodule $G$ of $F$ containing $S$ with $G\in \mathcal{F}$ and $\sharp (G)\leq \kappa_{\alpha}$. Then $M$ has an $\mathcal{F}$-preenvelope.}
\end{lemma}

\begin{lemma}\cite[Corollary 6.3.5]{EJ00}\label{212}
{\rm
Let $\mathcal{F}$ be a class of $R$-modules that is closed under direct summands and inverse limits, and $M$ be an $R$-module. If $M$ has an $\mathcal{F}$-preenvelope, then it has an $\mathcal{F}$-envelope.}
\end{lemma}

\begin{proposition} Let $\Lu$  be the class of all Lucas modules over a ring $R$. Then every $R$-module has an $\Lu$-envelope.
\end{proposition}
\begin{proof}
Let $M$ be an $R$-module with $\sharp(M)=\kappa_{\beta}$. For any Lucas module $F$ and $S\subseteq F$, by Lemma  \ref{210}, there exists a cardinal number $\kappa_{\alpha}$ depended only on $S$ and $R$ and a pure submodule $G $ of $F$ such that
$S\subseteq G$ and $\sharp(G) \leq \kappa_{\alpha}$. Then, by Proposition \ref{l-p}, $G$ also is a Lucas module. By Lemma \ref{211}, $M$ has an $\Lu$-preenvelope. By Proposition \ref{l-p} and Theorem \ref{invse lim lucas}, every direct summand and inverse system of Lucas modules are Lucas modules. Therefore every $R$-module has an $\Lu$-envelope by Lemma \ref{212}.
\end{proof}

Let $M$ be an $R$-modules. Then there exists exactly one $\Lu$-envelope of $M$ up to isomorphism. Next, we will study the construction of $\Lu$-envelope of an $R$-module.

\begin{lemma}\label{305}
{\rm
Let $M$ be a $\Q$-torsion-free $R$-module. Then the embedding map $\lambda: M\rightarrow M_q$ is the $\Lu$-envelope of $M$.}
\end{lemma}
\vskip1mm

\begin{proof} Let $L$ be a Lucas module. Let $M$ be a $\Q$-torsion-free $R$-module. Then $M_q/M$ is $\Q$-torsion by \cite[Proposition 3.5]{wzcc20}. So $\Ext^1_R(M_q/M,L)=0$ by Proposition \ref{c-l-p}. Consider the exact sequence
$$\Hom_R(M_q,L)\rightarrow \Hom_R(M,L)\rightarrow \Ext^1_R(M_q/M,L)=0.$$
So $\lambda :M\rightarrow M_q$ is an $\Lu$-preenvelope of $M$.

Let  $g:M_q\rightarrow M_q$ be an $R$-homomorphism satisfying $\lambda=g\circ \lambda$. Since $M_q\subseteq E(M)$, we have  $\lambda=g\circ \lambda$ is an essential extension of $M$. Hence, $g$ is a monomorphism. So $\im g \cong M_q$, and hence $\im g$ is a Lucas module. Since $M_q$ is the smallest Lucas module containing  $M$ by \cite[Proposition 2.10]{wzcc20}, it follows that $\im g=M_q$. So $g$ is an epimorphism.  Hence, $g$ is also an isomorphism.
\end{proof}

Next we will give the $\Lu$-envelopes for all $R$-modules.
\begin{theorem}\label{306}
{\rm
Let $M$ be an $R$-module. Then the natural composition $$\phi:M\twoheadrightarrow M/\Tor_{\Q}(M)\hookrightarrow (M/\Tor_{\Q}(M))_q$$ is the  $\Lu$-envelope of $M$. Moreover, $\phi$ has the unique mapping property.}
\end{theorem}

\begin{proof} Set $T=\Tor_{\Q}(M)$. Consider the exact sequences
$0\rightarrow T\rightarrow M\rightarrow M/T\rightarrow 0$ and $0\rightarrow M/T\rightarrow (M/T)_q\rightarrow K\rightarrow 0$. Then $K$ is  $\Q$-torsion by \cite[Proposition 3.5]{wzcc20}.
Let  $L$ be a Lucas module, we have the following two exact sequences
$$0\rightarrow \Hom_R(M/T,L)\rightarrow \Hom_R(M,L)\rightarrow \Hom_R(T,L)=0$$ and
$$0=\Hom_R(K,L)\rightarrow \Hom_R((M/T)_q,L)\rightarrow \Hom(M/T,L)\rightarrow \Ext_R^1(K,L)=0.$$
So the sequence $\Hom_R((M/T)_q,L)\rightarrow \Hom_R(M,L)\rightarrow 0$ is exact. Hence  $\phi$ is an $\Lu$-preenvelope of $M$.

Let $g:(M/T)_q\rightarrow (M/T)_q$ be an $R$-homomorphism satisfying $\phi=g\circ \phi$. Setting $f:M\twoheadrightarrow M/\Tor_{\Q}(M)$ and $h:M/\Tor_{\Q}(M)\hookrightarrow (M/\Tor_{\Q}(M))_q$ to be the natural map, then  $h\circ f=g\circ h\circ f$. Since $f$ is an epimorphism, it follows that  $h=g\circ h$.  By Lemma \ref{305}, we have $g$ is an isomorphism. So $\phi$ is an $\Lu$-envelope of $M$.

Now let $F'$ be a Lucas $R$-module. Then as  above we have the following natural isomorphisms: $$\Hom_R(M,F')\cong \Hom_R(M/T,F')\cong  \Hom_R((M/T)_q,F').$$
Hence for any $R$-homomorphism $f':M\rightarrow F'$, there exits a unique $R$-homomorphism $\varphi:(M/T)_q\rightarrow F'$ such that $\phi=\varphi\circ f'$. That is,  $\phi$ has the unique mapping property.
\end{proof}

Let $M$ be an $R$-module. Then we always denote $(M/\Tor_{\Q}(M))_q=L_q(M)$, and
call the natural map $\phi:M\rightarrow L_q(M)$ to be the Lucas envelope of $M$.

\begin{corollary}\label{int-lu-en}
{\rm
Let $R$ be an integral domain with $Q$ its quotient field, and $\mathcal{V}$ the class of all vector spaces over $Q$.
Then the following statements hold.
\begin{enumerate}
\item $\mathcal{V}$ is closed under its  pure submodule and inverse limit.
\item every   $R$-module $M$ has a $\mathcal{V}$-envelope $\phi: M\rightarrow M\otimes_RQ$ with unique mapping property.
\end{enumerate}
 }
\end{corollary}
\begin{proof} The two statements follow by Proposition \ref{l-p}, Proposition \ref{int-luc}, Theorem \ref{invse lim lucas} and Theorem \ref{306}.
\end{proof}

The following example shows that the inverse limit of $R$-modules may not commutate with their Lucas   envelopes, that is, $L_q(\lim\limits_{\longleftarrow}M_i)\not\cong\lim\limits_{\longleftarrow}L_q(M_i)$ in general.
\begin{example}\label{ex-inv-luc-env}
{\rm
Let $\mathbb{Z}$ be the ring of all integers with its quotient field $\mathbb{Q}$, $\mathbb{P}$ all primes in $\mathbb{Z}$ and $\mathbb{I}_p$ the quotient of $\mathbb{Z}$ by the ideal generated by a prime $p\in \mathbb{P}$. Since $\mathbb{I}_p$ is a torsion module, the  Lucas   envelope $L_q(\mathbb{I}_p)\cong \mathbb{I}_p\otimes_{\mathbb{Z}}\mathbb{Q}=0$ by \cite[Chapter I, Exercise 6.7]{FS01}. So $\prod_{p\in\mathbb{P}}L_q(\mathbb{I}_p)=0$. However, $L_q(\prod_{p\in\mathbb{P}}\mathbb{I}_p)$ is not equal to $0$ since $(\overline{1})\in \prod_{p\in\mathbb{P}}\mathbb{I}_p$ is not a torsion element (see \cite[Chapter I, Exercise 6.7]{FS01} again). Hence $\prod_{p\in\mathbb{P}}L_q(\mathbb{I}_p)\not\cong L_q(\prod_{p\in\mathbb{P}}\mathbb{I}_p)$.
}
\end{example}

\section{the direct limits and covering properties of Lucas modules}

We begin with the following result.
\begin{lemma}\label{dir lim Q-tf} Let $\{M_i, \varphi_{i,j}\mid i,j\in \Gamma\}$ be a direct system of $\Q$-torsion free  modules over $R$. Then $\lim\limits_{\longrightarrow}M_i$ is also $\Q$-torsion free.
\end{lemma}
\begin{proof} Let $I\in \Q$ and $\{M_i, \varphi_{i,j}\mid i,j\in \Gamma\}$  a direct system of $\Q$-torsion free  modules. Then $\Hom_R(R/I,\lim\limits_{\longrightarrow}M_i)\cong \lim\limits_{\longrightarrow}\Hom_R(R/I,M_i)=0$ by \cite[Lemma 2.7]{gt}. So $\lim\limits_{\longrightarrow}M_i$ is also $\Q$-torsion free.
\end{proof}

\begin{proposition}\label{dir lim iso Q-tf} Let $I\in \Q$ and $\{M_i, \varphi_{i,j}\mid i,j\in \Gamma\}$ be a direct system of $\Q$-torsion free  $R$-modules. Then we have a natural isomorphism: $$\lim\limits_{\longrightarrow}\Hom_R(I,M_i)\cong\Hom_R(I, \lim\limits_{\longrightarrow}M_i).$$
\end{proposition}
\begin{proof} Consider the exact sequence $0\rightarrow K\rightarrow F\rightarrow I\rightarrow 0$ with $F$ a finitely generated free $R$-module. Then we have the following commutative diagram with rows exact:
$$\xymatrix@R=20pt@C=15pt{
0 \ar[r]^{}  & \lim\limits_{\longrightarrow}\Hom_R(I,M_i) \ar[d]_{f_1}\ar[r]^{} & \lim\limits_{\longrightarrow}\Hom_R(F,M_i)  \ar[d]^{f_2}\ar[r]^{} & \lim\limits_{\longrightarrow}\Hom_R(K,M_i)\ar[d]^{f_3}    \\
   0 \ar[r]^{} & \Hom_R(I, \lim\limits_{\longrightarrow}M_i)\ar[r]^{} &\Hom_R(F, \lim\limits_{\longrightarrow}M_i)  \ar[r]^{} & \Hom_R(K, \lim\limits_{\longrightarrow}M_i) .}$$
Trivially, $f_2$ is an isomorphism and $f_1$ is a monomorphism by \cite[Lemma 2.7]{gt}. To show $f_1$ is an epimorphism, we just need to prove  $f_3$ is a monomorphism. Indeed, since $I\in\Q$, it follows that $I[x]$ is a regular ideal of $R[x]$ by \cite[Exercise 6.5]{fk16}. Then $I\otimes_R\T(R[x])=\T(R[x])$ is a finitely presented $\T(R[x])$-module. So $I$ is $\tau_q$-finitely presented by \cite[Theorem 3.3]{z-tq-f-c}. Hence, $K$ is  $\tau_q$-finitely generated by \cite[Remark 3.2(5)]{z-tq-f-c}. It follows that there is a finitely generated submodule $L$ of $K$ such that $K/L$ is $\Q$-torsion. Consider the following commutative diagram with rows exact:
$$\xymatrix@R=20pt@C=15pt{
0 \ar@{=}[r]^{}  & \lim\limits_{\longrightarrow}\Hom_R(K/L,M_i) \ar[d]_{g_1}\ar[r]^{} & \lim\limits_{\longrightarrow}\Hom_R(K,M_i)  \ar[d]^{f_3}\ar[r]^{} & \lim\limits_{\longrightarrow}\Hom_R(L,M_i)\ar[d]^{g_2}    \\
   0 \ar@{=}[r]^{} & \Hom_R(K/L, \lim\limits_{\longrightarrow}M_i)\ar[r]^{} &\Hom_R(K, \lim\limits_{\longrightarrow}M_i)  \ar[r]^{} & \Hom_R(L, \lim\limits_{\longrightarrow}M_i) .}$$
Since $L$ is finitely generated, $g_2$ is a monomorphism by \cite[Lemma 2.7]{gt}. Hence $f_3$ is also a monomorphism by Five Lemma.
\end{proof}

\begin{theorem}\label{dir lim lucas} Let $\{M_i, \varphi_{i,j}\mid i,j\in \Gamma\}$ be a direct system of Lucas  modules over $R$. Then $\lim\limits_{\longrightarrow}M_i$ is also a Lucas module.
\end{theorem}
\begin{proof} Let $I\in \Q$ and  $\{M_i, \varphi_{i,j}\mid i,j\in \Gamma\}$  a direct system of Lucas  modules over $R$. Consider the following commutative diagram with rows exact:
$$\xymatrix@R=20pt@C=15pt{
\lim\limits_{\longrightarrow}\Hom_R(R,M_i) \ar[d]_{\cong}\ar[r]^{} & \lim\limits_{\longrightarrow}\Hom_R(I,M_i)  \ar[d]^{f_1}\ar[r]^{} & \lim\limits_{\longrightarrow}\Ext^1_R(R/I,M_i)\ar[d]^{h_1}\ar[r]^{} & 0   \\
\Hom_R(R, \lim\limits_{\longrightarrow}M_i)\ar[r]^{} &\Hom_R(I, \lim\limits_{\longrightarrow}M_i)  \ar[r]^{} &\Ext^1_R(R/I, \lim\limits_{\longrightarrow}M_i) \ar[r]^{} & 0.}$$
 By Proposition \ref{dir lim iso Q-tf}, $f_1$ is an isomorphism. So $h_1$ is also an isomorphism. Hence $\Ext^1_R(R/I, \lim\limits_{\longrightarrow}M_i) \cong \lim\limits_{\longrightarrow}\Ext^1_R(R/I,M_i)=0$. Combining with Lemma \ref{dir lim Q-tf}, we have  $\lim\limits_{\longrightarrow}M_i$ is also a Lucas module.
\end{proof}

\begin{lemma}\label{LC-enve-loc} Let $\m$ be a maximal  $q$-ideal of $R$ and $M$  an $R$-module. Then $L_q(M)_{\m}\cong M_{\m}$.
\end{lemma}
\begin{proof} Let $M$ be an $R$-module. Then we have two exact sequences $0\rightarrow tor_{\Q}(M)\rightarrow M\rightarrow M/tor_{\Q}(M)\rightarrow 0$ and $0\rightarrow  M/tor_{\Q}(M)\rightarrow L_q(M)\rightarrow T\rightarrow 0$. Then $T$ is $\Q$-torsion by \cite[Proposition 3.5]{wzcc20}.  Let  $\m$ be a maximal  $q$-ideal of $R$. Localizing at $\m$, we have  $L_q(M)_{\m}\cong M_{\m}$ by \cite[Theorem 3.4]{wzcc20}.
\end{proof}

\begin{proposition}\label{dir lim-LC-enve} Let $\{M_i, \varphi_{i,j}\mid i,j\in \Gamma\}$ be a direct system of $R$-modules. Then $L_q(\lim\limits_{\longrightarrow}M_i)\cong\lim\limits_{\longrightarrow}L_q(M_i)$.
\end{proposition}
\begin{proof} Let $\m$ be a maximal  $q$-ideal of $R$. Then we have  $$(L_q(\lim\limits_{\longrightarrow}M_i))_{\m}\cong (\lim\limits_{\longrightarrow}M_i)_{\m}\cong \lim\limits_{\longrightarrow}(M_i)_{\m}\cong \lim\limits_{\longrightarrow}(L_q(M_i))_{\m}$$
by Lemma \ref{LC-enve-loc}. It follows  by \cite[Proposition 2.7(6)]{ZDC20} that $$L_q(\lim\limits_{\longrightarrow}M_i)\cong (L_q(\lim\limits_{\longrightarrow}M_i))_q\cong(\lim\limits_{\longrightarrow}L_q(M_i))_q\cong \lim\limits_{\longrightarrow}L_q(M_i)$$
since $\lim\limits_{\longrightarrow}L_q(M_i)$ is a Lucas module by Theorem \ref{dir lim lucas}.
\end{proof}

Let  $\mathcal{C}$ be a  class  of $R$-modules and $M$ an $R$-module. An $R$-homomorphism $f: C\rightarrow M$ with  $C\in \mathcal{C}$ is said to be a $ \mathcal{C}$-precover  provided that  for any $C'\in \mathcal{C}$, the natural homomorphism  $\Hom_R(C',C)\rightarrow \Hom_R(C',M)$ is an epimorphism. If, moreover, every endomorphism $h$ such that $f=f\circ h$  is an isomorphism, then $f: C\rightarrow M$ is said to be a
$\mathcal{C}$-cover.  If every $R$-module has a $\mathcal{C}$-precover (resp.,  $\mathcal{C}$-cover), then $\mathcal{C}$ is said to be precovering (resp., $\mathcal{C}$-covering). Let $f: C\rightarrow M$ be a  $\mathcal{C}$-cover of $M$. If for any $f':C'\rightarrow M$ with $C'\in\mathcal{C}$, there exists a unique $g:C'\rightarrow C$ such that $f'=f\circ g$. Then $f$ is said to have the unique mapping property. Let $\Lu$  be the class of all Lucas modules over a ring $R$. An $\Lu$-(pre)cover of an $R$-module $M$ is said to be a Lucas (pre)cover of $M$. The author et al. \cite{zxl20} put \cite[Theorem 3.4]{DD19},  {\cite[Theorem 3.4.8]{P09} and {\cite[Theorem 2.5]{HJ08}} together in the following lemma.

\begin{lemma}\cite[Lemma 3.4]{zxl20}\label{definable}
Let  $\mathcal{C}$ be a class  of $R$-modules that is closed under pure submodules and direct products. Then the following assertions are equivalent:
\begin{enumerate}
    \item $\mathcal{C}$ is closed under direct limits, that is, $\mathcal{C}$ is definable;
    \item $\mathcal{C}$ is closed under pure quotients;
    \item $\mathcal{C}$ is covering;
     \item $\mathcal{C}$ is precovering.
\end{enumerate}
\end{lemma}

\begin{theorem}\label{lu-c} Let $R$ be a ring. Then the following statements hold.
\begin{enumerate}
\item every pure quotient of a Lucas module is a Lucas module.
\item every $R$-module has a Lucas cover.
\end{enumerate}
\end{theorem}
\begin{proof} By Proposition \ref{l-p}, the class of Lucas modules is closed under  pure submodules and direct products. So $(1)$ and $(2)$ hold by Theorem \ref{dir lim lucas} and Lemma \ref{definable}.
\end{proof}

In the rest of this article, we consider when every $R$-module has a Lucas cover with unique mapping property.

Recall that a ring $R$ is called a $\WQ$-ring if $R$ itself is a Lucas $R$-module, or equivalently $R=Q_0(R)$, where $Q_0(R):=\{\alpha\in T(R[x])\mid\ \mbox{there exists}\ I\in \Q\ \mbox{such that } I\alpha\subseteq R\}$. So an integral domain is a $\WQ$-ring if and only if it is a field.  Now we characterize $\WQ$-rings over which every module has a Lucas cover with the unique mapping property.

\begin{proposition}\label{WQump} Let $R$ be a $\WQ$-ring. If every $R$-module has a Lucas cover with the unique mapping property, then $R$ is a $\DQ$-ring.
\end{proposition}
\begin{proof}  Suppose $R$ is a $\WQ$-ring, then every projective $R$-module is a Lucas module. So every Lucas cover is an epimorphism. Let $M$ be an $R$-module and $0\rightarrow K\rightarrow C\xrightarrow{f} M\rightarrow 0$ be an exact sequence with $f$ the Lucas cover of $M$. Let $F$ be a free $R$-module with the rank $\aleph$ which is greater than the card of the generators of $M$. Then $F$ is a Lucas module, and so we have the following exact sequence $$0\rightarrow \Hom_R(F,K)\rightarrow \Hom_R(F,C)\rightarrow \Hom_R(F,M)\rightarrow 0.$$
Since $f$ has the unique mapping property, the natural map $\Hom_R(F,C)\rightarrow \Hom_R(F,M)$ is a monomorphism. Hence $\Hom_R(F,K)\cong \prod\limits_{\aleph} \Hom_R(R,K)=0$, and so $K=0.$ It follows that $M\cong C$ is a Lucas $R$-module. Consequently, $R$ is a $\DQ$-ring.
\end{proof}

\begin{example}\cite[Example 12]{L93}
Let $D=L[X^2,X^3,Y]$, $\Pp=Spec(D)-\{\langle X^2,X^3,Y\rangle\}$, $B=\bigoplus\limits_{\p\in \Pp}K(R/\p)$  and $R=D(+)B$ where $L$ is a field and $K(R/\p)$ is the quotient field of $R/\p$. Since $Q_0(R)=R$, $R$ is a $\WQ$-ring by \cite[Proposition 3.8]{wzcc20}. Note that every finitely generated $R$-ideal of the form $J(+)B$ with $\sqrt{J}=\langle X^2,X^3,Y\rangle$ is a semi-regular. Hence $R$ is not a $\DQ$-ring by \cite[Proposition 2.2]{fkxs20}. So by Proposition \ref{WQump} and its proof, every non-Lucas $R$-modules satisfies that their Lucas cover does not have the unique mapping property.
\end{example}

We show that the Lucas covers of modules over integral domains have the unique mapping property.

\begin{corollary}\label{lu-c-v} Let $R$ be an integral domain with its quotient field $Q$, and $\mathcal{V}$ the class of all vector spaces over $Q$. Then the following statements hold.
\begin{enumerate}
\item $\mathcal{V}$ is closed under its  pure quotient and direct limit.
\item every $R$-module has a  $\mathcal{V}$-cover with the unique mapping property.
\end{enumerate}
\end{corollary}
\begin{proof} We just need to show every $\mathcal{V}$-cover has the unique mapping property since the others follow from Proposition \ref{int-luc}, Theorem \ref{dir lim lucas} and Theorem \ref{lu-c}.  Let $f:V\rightarrow M$ be a $\mathcal{V}$-cover of $M$. Then by \cite[Lemma 5.8]{gt}, $\Ker(f)$ has no $Q$-subspace of $V$. We claim $\Hom_R(V',\Ker(f))=0$. Indeed, let $h:V'\rightarrow \Ker(f)$ be an $R$-homomorphism. Consider the composition $\phi: V'\xrightarrow{h} \Ker(f)\xrightarrow{i}V$. Since $\Hom_D(V',V)\cong \Hom_Q(V',V)$, we have $\Im(\phi)=\Im(h)$ is a $Q$-subspace of $V$, and hence equal to $0$. So $h$ is an zero map. Hence $\Hom_R(V',\Ker(f))=0$. Considering the following exact sequence: $$0=\Hom_R(V',\Ker(f))\rightarrow \Hom_R(V',V)\rightarrow \Hom_R(V',M),$$ we have $f$ is a $\mathcal{V}$-cover with the unique mapping property.
\end{proof}

\begin{acknowledgement}\quad\\
The second author was supported by  National Natural Science Foundation of China (No. 11901412). The third author was supported by  National Natural Science Foundation of China (No. 12201361).
\end{acknowledgement}

\end{document}